\newif\ifElsevier

\Elsevierfalse 	

\ifElsevier
	\documentclass{elsarticle}
\else
	\documentclass[12pt]{article}
\fi

\usepackage[margin=1in]{geometry}
\usepackage{amsmath,amssymb,amsthm,graphicx}
\usepackage{subcaption}
\usepackage{enumitem}
\usepackage{todonotes}
\usepackage{multirow}
\usepackage{mathtools}
\usepackage{hyperref}
\usepackage[ruled,linesnumbered]{algorithm2e}
\usepackage{mathrsfs}
\usepackage{tikz}
\usetikzlibrary{decorations.pathreplacing}
\usetikzlibrary{arrows.meta}
\usepackage{pgfplots,pgfplotstable}
\usepackage[multiple]{footmisc}
\usepackage[T1]{fontenc}
\usepackage[utf8]{inputenc}
\usepackage{url}
\usepackage{breakurl}

\newtheorem{theorem}{Theorem}

\newtheorem{lemma}[theorem]{Lemma}
\newtheorem{corollary}[theorem]{Corollary}

\newtheorem{proposition}[theorem]{Proposition}

\newtheorem{observation}[theorem]{Observation}
\newtheorem{remark}[theorem]{Remark}

\DeclarePairedDelimiter{\card}{\lvert}{\rvert}

\newcommand{\Z}{\mathbb{Z}}

\newcommand{\R}{\mathbb{R}}

\newcommand{\psd}{\succcurlyeq 0}
\newcommand{\unaryminus}{\scalebox{0.7}[1.0]{\( - \)}}

\newcommand{\Sy}{{\mathcal{S}_n}}
\newcommand{\met}{{\textup{MET}}}
\newcommand{\hs}{\{\unaryminus 1,1\}^{n}}
\newcommand{\hsp}{\{\unaryminus 1,1\}^{n+1}}

\newcommand{\Las}{{\textup{Las}}}
\newcommand{\GW}{{\textup{GW}}}
\newcommand{\CLI}{{\textup{CLI}}}
\newcommand{\expedis}{{\texttt{EXPEDIS}}}
\newcommand{\zmc}{{z_{\mathrm{max-cut}}}}
\newcommand{\bolddot}{\boldsymbol{\cdot}}

\newcommand{\rew}{\textcolor{black}} 
\newcommand{\reww}{\color{black}}
\newcommand{\rewt}{\textcolor{black}} 

\newcommand{\specificthanks}[1]{\@fnsymbol{#1}}

\DeclareMathOperator{\tr}{tr}				
\DeclareMathOperator{\Diag}{Diag}			
\DeclareMathOperator{\diag}{diag}			
\DeclareMathOperator{\rank}{rk}				
\DeclareMathOperator{\argmin}{argmin}		        
\DeclareMathOperator{\ns}{null}				

\begin{document}

\ifElsevier
\else
	\title{\expedis: An Exact Penalty Method over Discrete Sets}
\fi

\ifElsevier
\else	

\author{Nicol{\`o} Gusmeroli\thanks{Alpen-Adria-Universit{\"a}t Klagenfurt, Universit\"atsstra{\ss}e 65-67, 9020 Klagenfurt, Austria. 
\newline Emails: \{nicolo.gusmeroli,angelika.wiegele\}@aau.at} \thanks{\rew{Part of this work has been done
        while the first author was employed at TU Dortmund.}}
\and{Angelika Wiegele}\footnotemark[1]}

\fi

\ifElsevier
\else
  \maketitle
\fi

\begin{abstract}
  We address the problem of minimizing a quadratic function subject to
  linear constraints over binary variables. We introduce the exact solution
  method called \expedis\; where the constrained problem is transformed into a max-cut
  instance, and then the whole machinery available for max-cut can be
  used to solve the transformed problem.
  We derive the theory in order to find a transformation in the spirit
  of an exact penalty method; however, we are only interested in
  exactness over the set of binary variables.
  In order to compute the maximum cut we use the solver BiqMac.
  Numerical results show that
  this algorithm can be successfully applied on various classes of
  problems.  
\end{abstract}

\section{Introduction}\label{sec:INT}
	
We address the problem of solving linearly constrained binary quadratic
problems of the following form:
\begin{equation}\label{prob:BQP01}\tag{{BQP$_{01}$}}
f^* = \min \big\{ \hat{f}(y) = y^\top \hat{F}y + \hat{c}^\top
y \mid \hat{A}y = \hat{b}, \ y \in \{0,1\}^n \big\},
\end{equation}
where $\hat{F} \in \R^{n \times n}$ is a symmetric matrix,
$\hat{c} \in \R^n$, and the linear equations are given via  
$\hat{A} \in \Z^{m \times n}$ and $\hat{b} \in \Z^m$.

Problem~\eqref{prob:BQP01} encompasses  0/1 linear programming
problems and unconstrained quadratic 0/1 problems, which are both
known to be  classes of NP-hard problems. Several well-known NP-hard
problems from combinatorial optimization, like max-cut, stable set,
graph partitioning, graph coloring, routing problems, knapsack
problems etc. are explicit instances of these two classes, see, 
e.g.,~\cite{GaJoh:79,NeWo:88,Wa:14} 
for definitions and proofs. 

All these problems have a wide range of applications and there is big
interest in solution methods for \eqref{prob:BQP01} also outside the
mathematical optimization area. 
Data science (clustering analysis),  logistics (quadratic assignment
problem, vehicle routing problem), telecommunications (several
versions of frequency assignment problem), finance (portfolio
optimization problem), etc. are some of the areas where solving the
underlying linear or quadratic 0/1 problems is essential, see, e.g.,
the survey papers~\cite{GeGu:2011,KoHaLe:2014}. 

Solving Problem~\eqref{prob:BQP01} to optimality is always highly
appreciated. Even when good solutions based on appropriate (meta)heuristics 
are acceptable for practical needs, the developers of such
algorithms still need to evaluate them and this can be done only if
optimal solutions on problems of (at least) medium size are
available. 

Global optimization solvers that can handle problems of this type are
typically branch-and-bound algorithms. One can group them according to
the 
different types of relaxations used in order to obtain lower
bounds. Among them are relaxations based on
reformulation-linearization techniques (RLT)~\cite{ShAd:90} but also
relaxations based on semidefinite programming (SDP) have been
successfully implemented~\cite{PoReWo:95, BuVa:08, BuWi:13}.

In this paper we follow the idea introduced by J.~B.~Lasserre
in~\cite{La16}. We will reformulate Problem~\eqref{prob:BQP01} as a
max-cut problem, which we then solve using the solver BiqMac developed
by Rendl, Rinaldi and Wiegele~\cite{ReRiWi10}.
The crucial part is to find a penalty parameter used in the
transformation large enough to get 
equality of the two problems but at the same time to be kept small in
order to not run into numerical difficulties.

The max-cut problem is a well-studied combinatorial optimization
problem. Hence, the whole machinery developed for max-cut can be
used in order to solve the underlying problem.
Transformation to max-cut is also beneficial since quantum annealers
like D-Wave systems ask for max-cut problems as input type.

We introduce algorithm \expedis\ that computes a penalty parameter and
solves the transformed problem using the solver BiqMac.
We derive the theory on what is necessary to get such a transformation
to a max-cut instance. In particular, we state conditions on a minimal
penalty parameter. Several variants of
how the parameter can be computed as well as refinements with respect
to infeasibility or known feasible solutions are presented.
Numerical results demonstrate that this procedure works well on 
\rew{randomly generated instances}
instances as well as on several classes of instances from the
literature, like the max $k$-cluster problem.

The remainder of this paper is structured as follows: in
Section~\ref{sec:PRE} we briefly describe a well-known procedure to
derive relaxations based on semidefinite programming (SDP) and we give
a formulation of the max-cut problem together with a short explanation
of the exact solution method BiqMac; in Section~\ref{sec:EXPA} we begin the
heart of this paper -- we describe Algorithm~\expedis, an exact penalty
method over discrete sets; in
Section~\ref{sec:LAS} we show that \expedis\ is a generalization of the
method introduced by Lasserre~\cite{La16}; Section~\ref{sec:PARAM}
states the necessary conditions on how to choose the parameters used
in \expedis\ while
in Section~\ref{sec:IMP} we give recipes on computing them;
in Section~\ref{sec:REF} refinements of the algorithm are discussed
before we present our numerical results in Section~\ref{sec:RES};
Section~\ref{sec:CON} concludes this paper giving a summary and an
outlook on future research. 

	\paragraph{Notation} 
        We denote by $e$ the vector of all ones, by $J$ the matrix of all ones, and by $e_j$ the unit vector with value~1 in the $j$-th component and~0 everywhere else.
        The 0-vector and the 0-matrix is denoted by $\mathbf{0}$. 
        Given a matrix $A$, $A_{i, \bolddot }$ is row $i$ of $A$ and $A_{\bolddot, j}$ is column $j$ of $A$.
        Matrix $A \in \R^{n \times n}$ is positive semidefinite, denoted $A \psd $, if $x^\top Ax \geq 0$ for every $x \in \R^n$ and 
        for $A,B \in \R^{n \times n}$, we define the inner product $\langle A,B \rangle = \tr\left(B^\top A\right) = \sum_i \sum_j A_{ij} B_{ij}$.
        The vector holding the diagonal elements of a matrix $X \in \R^{n \times n}$ is denoted by $\diag(X)$, i.e., $\diag(X)_i = X_{ii}$, and 
        given a vector $x \in \R^n$, by $\Diag(x)$ we denote the $n \times n $ diagonal matrix $X$ such that its diagonal vector is $x$.

\section{Preliminaries}\label{sec:PRE}

        \subsection{Solving Max-Cut Problems}\label{sec:biqmac}
        The max-cut problem is among the most studied combinatorial optimization problems. It has connections to various fields of discrete mathematics and has a wide range of applications. 
        Max-cut is an NP-hard problem, but several approximation algorithms as well as exact methods using some kind of branch-and-bound type methods exist, see, e.g., \cite{ReRiWi10} for more details and references.
        As several other combinatorial optimization problems, max-cut problems are easy to state. Let be given an undirected graph $G=(V,E)$, having vertex set $V$ and edge set $E$ with weights $w_e\in \R$ on the edges $e\in E$.
        The max-cut problem asks to partition the vertex set into two parts $(S, V\backslash S)$ in a way such that the sum of the weights on the edges having exactly one endpoint in $S$ is maximized, i.e., we look for a subset of the edges  
        \begin{equation*}
          \delta(S) = \{ e= uv\in E\mid u \in S, v\not\in S\}
        \end{equation*}
        where $S\subseteq V$, such that $\sum_{e \in \delta(S)} w_e$ is maximized. Let $A=(a_{ij})$ be the adjacency matrix of the graph, i.e., $a_{ij} = w_e$ for $e=\{i,j\}$. The Laplace matrix of the graph associated with $A$ is given as 
        \begin{equation*}
          L = \rew{\Diag}(Ae) - A
        \end{equation*}
        and defines $C=\frac{1}{4}L$.
        Then we can find the maximum cut by solving the binary quadratic problem
        \begin{equation}\label{eq:zmc}
          \zmc = \max \left\{ x^\top Cx \mid x\in \{\unaryminus 1,1\}^{\card{V}} \right\}.
        \end{equation}

        Among the most efficient solvers for computing the maximum cut in a (medium-sized) graph is \emph{BiqMac}~\cite{ReRiWi10}.
        BiqMac uses semidefinite relaxations in order to generate high quality upper bounds on the maximum cut. In particular, the approximate solution of the semidefinite relaxation 
        \begin{equation}\label{eq:sdp-met}
          \max \left\{ \langle C,X \rangle \mid X \psd, \ \diag(X) = e,\ X\in \met \right\}
        \end{equation}
        serves as an upper bound in a branch-and-bound scheme (see Section~\ref{sec:sdp}).

        In order to derive a lower bound (finding a cut in the
        graph with a large value), the Goemans-Williamson hyperplane
        rounding technique~\cite{GoWi95} is applied to the matrix
        obtained by solving the SDP~\eqref{eq:sdp-met}. All details
        about the BiqMac algorithm can be found in~\cite{ReRiWi10}.

	\subsection{Semidefinite Relaxations of Binary Problems}\label{sec:sdp}

	There is a well-known procedure on how to derive semidefinite
        relaxations for the $\pm 1$ version of
        problem~\eqref{prob:BQP01}. (See Problem~\eqref{prob:BQP} in
        Section~\ref{sec:EXPA} for an explicit formulation in the $\pm
        1$ setting.)
        The following equivalence is easy to prove.
	\begin{equation}\label{eq:rank1formulation}
	\big\{xx^\top: x \in \hs \big\} =  \{ X \in \Sy \mid X \psd, ~\diag(X)=e, ~\rank(X)=1\}
	\end{equation} 
	Thus, problem~\eqref{prob:BQP} has an equivalent formulation as
	\begin{equation*}
	f^*= \min \left\{ \langle F,X \rangle + c^\top x + \alpha \mid Ax = b, \ \begin{bmatrix*} 1 & x^\top \\ x & X \end{bmatrix*} \psd, \ \diag(X) = e, \ \rank(X) = 1 \right\}.
	\end{equation*} 
        In this formulation, all non-convexity (from the objective function as well as from the binary conditions) is hidden in the rank-1 constraint. 
        Hence, it is straightforward to derive a semidefinite relaxation by dropping the rank-condition, 
	\begin{equation}\label{prob:SDP}
	\min \left\{ \langle F,X \rangle + c^\top x + \alpha \mid Ax = b, \ \begin{bmatrix*} 1 & x^\top \\ x & X \end{bmatrix*} \psd, \ \diag(X) = e \right\}.
	\end{equation} 
	In the absence of linear constraints, this is the \emph{Shor relaxation}~\cite{Sh87}. It can be \rew{solved} in polynomial time using, e.g., interior point methods. 
        For a more detailed study about semidefinite programming, we refer the reader to the handbooks~\cite{handbook-2012, handbook-2000} and the references therein.

        \paragraph{Adding cutting planes.}
        Relaxation~\eqref{prob:SDP} can be tightened by adding polyhedral cuts.
        In particular, {\em clique inequalities}~\cite{LaPo96}
        turn out to strengthen relaxation~\eqref{prob:SDP} significantly. 

        Consider the vector $b$  with entries from the set
        $\{\unaryminus 1,0,1\}^n$ and an odd number of nonzero
        entries. Then
        \begin{equation*}
          \min \left\{ (b^\top x)^2 \mid x \in \hs \right\} = 1,
        \end{equation*}
        hence the inequality $b^\top Xb \ge 1$ is valid for
        Problem~\eqref{prob:BQP} and can be used to tighten 
        relaxation~\eqref{prob:SDP}. 

        When the vector $b$ consists of three nonzero entries, the
        arising clique inequalities are the so called \emph{triangle
          inequalities}, i.e., the set of constraints 
        \begin{align*}
          x_{ij} + x_{ik} + x_{jk} &\ge -1\\
          x_{ij} - x_{ik} - x_{jk} &\ge -1\\
         -x_{ij} + x_{ik} - x_{jk} &\ge -1\\
         -x_{ij} - x_{ik} + x_{jk} &\ge -1
        \end{align*}
        for all $1\le i < j < k\le n$. The polytope containing all $X$ that
        satisfy these triangle inequalities is called the metric
        polytope and is denoted by $\met$.
        Adding these constraints tightens the SDP relaxation
        significantly. However, solving this strengthened SDP comes
        with a serious computational effort.  
        In~\cite{FiGrReSo06, ReRiWi10} a method to deal with such an
        SDP with a huge number of linear constraints has been
        developed. A (dynamic version) of a bundle method 
        is used in
        order to obtain an approximate solution, giving a safe upper
        bound on the maximum cut of the graph. 

        In case all triangle inequalities are satisfied, we
        can achieve 
        a further strengthening by considering vectors
        $b$ with five nonzero entries, leading to \emph{5-clique inequalities}.
        Differently from the triangle inequalities, the 5-clique inequalities are
        too many
        to be enumerated hence we use a heuristic to separate
        them.
        
        In this separation algorithm we create a set of random permutations of five elements. 
        Then we run a minimization problem over this permutation for a finite number of swaps.
        The swaps are accepted if the solution improves, and they are
        accepted with a certain probability if the solution does not
        improve. 
       	This procedure creates a set of 5-clique inequalities with a
        potentially high violation. 
        From this set, we add the most violated 5-clique inequalities
        to the relaxation.

	\section{Exact Penalty Method over Discrete Sets}\label{sec:EXPA}
	
	Given a symmetric matrix $\hat{F} \in \R^{n \times n}$ and a
        vector $\hat{c} \in \R^n$, we define the objective function
        $\hat{f}(y) = y^\top \hat{F}y + \hat{c}^\top
        y$. Moreover, we consider the linear equations 
        $\hat{A}y = \hat{b}$, where $\hat{A} \in \Z^{m \times n}$ and
        $\hat{b} \in \Z^m$. We want to find  
	\begin{equation}\label{prob:BQP01b}\tag{{BQP$_{01}$}}
	f^* = \min \big\{ \hat{f}(y) \mid \hat{A}y = \hat{b}, \ y \in \{0,1\}^n \big\},
	\end{equation}
        i.e., we want to solve a linearly constrained binary quadratic problem.
	
	We can reformulate problem~\eqref{prob:BQP01} to an equivalent
        formulation with variables in $\{\unaryminus 1,1\}$. Consider the change of
        variables $x = 2y -e$, and let $A = \frac{1}{2}\hat{A}$, $b
        = \hat{b} - \frac{1}{2}\hat{A}e$, $c =
        \frac{1}{2}(\hat{c}+\hat{F}e)$ and $F = \frac{1}{4}\hat{F}$ be the new
        parameters. Then 
	\begin{equation}\label{prob:BQP}\tag{BQP}
	f^* = \min \big\{ f(x) \mid A x = b, \ x \in \hs \big\}
	\end{equation}
	where $f(x) = x^\top F x + c^\top x + \alpha$ and
        $\alpha = \frac{1}{2} \hat{c}^\top e + \frac{1}{4}e^\top \hat{F} e
        $. In case problem~\eqref{prob:BQP} is infeasible, we have $f^*
        = + \infty$.  
        \begin{remark}\label{rem:axminusb}
        Note that since $\hat{A}$ and $\hat{b}$ are integer valued,
        for $y\in \{0,1\}^n$ the value of $\hat{A}y-\hat{b}$ is an
        integer as well. The transformation ensures $\hat{A}y - \hat{b} = Ax-b$. Therefore, for $x \in \hs$, the value of
        $Ax-b$ must also be integral, even though the values in $A$
        and $b$ might be fractional.  
        \end{remark}
	
	In order to simplify notation, we denote by $\Delta$ the set of feasible points of Problem~\eqref{prob:BQP}, and
        by $\Delta^c$ we denote the set of infeasible $\hs$ vectors, i.e., 
        \begin{align*}
          \Delta &= \big\{ x \in \hs \mid Ax = b \big\}\\
          \Delta^c &= \big\{ x \in \hs \mid Ax \not= b \big\}.
        \end{align*}
	We now introduce a penalty parameter $\sigma > 0$ and add a quadratic penalty function to $f(x)$, thus we have the function
	\begin{equation}
	h(x) = f(x) + \sigma \lVert Ax-b \rVert^2
	\end{equation} 
	and we consider the unconstrained binary quadratic problem
	\begin{equation}\label{prob:UBQP}\tag{UBQP}
	h^* = \min \big\{ h(x) \mid x \in \hs \big\}.
	\end{equation}
	Expanding terms, we can rewrite the objective function of Problem~\eqref{prob:UBQP}
	\begin{align*}
	h(x) &= f(x) + \sigma \lVert Ax-b \rVert^2 =\\
	&= x^\top F x +c^\top x + \alpha + \sigma (Ax-b)^\top (Ax-b) =\\
	&= x^\top (F + \sigma A^\top A)x + (c - 2\sigma A^\top b)^\top x + (\alpha + \sigma b^\top b) = \\
	&= \bar{x}^{\top} Q \bar{x}
	\end{align*}
	where $\bar{x} = \begin{bmatrix} 1\\ x\end{bmatrix}$ and 
        \begin{equation}\label{eq:Q}
          Q=\begin{bmatrix} \alpha + \sigma b^\top b & (c - 2\sigma A^\top b)^\top/2 \\ (c - 2\sigma A^\top b)/2 & F + \sigma A^\top A \end{bmatrix}.
        \end{equation}
	In this way we can restate Problem~\eqref{prob:UBQP} as
	\begin{equation}\label{prob:MC} \tag{MC}
	h^* = \min \big\{ \bar{x}^{\top} Q\bar{x}\mid \bar{x}\in \hsp, \bar{x}_0=1\big\} 
	\end{equation}
	which is a max-cut problem on a graph with $n+1$ vertices. To see this, we define $C = \rew{\Diag}(Q e) - Q$, which gives
	\begin{align*}
	  h^* &= \min \big\{ \bar{x}^{\top} Q \bar{x} \mid \bar{x} \in \hsp\big\} \\
          &= \unaryminus \max \big\{ \bar{x}^\top (-Q)\bar{x} \mid \bar{x} \in \hsp\big\}\\
          &=  e^\top Q e - \max \big\{ \bar{x}^\top C \bar{x} \mid \bar{x} \in \hsp\big\}
	\end{align*}
	where $C = \frac{1}{4}L$ and $L$ is the Laplace matrix of a graph with vertex set
        $\{0,1,\dots,n\}$ and adjacency matrix $\mathcal{A}$ with

	\begin{equation}\label{eq:adjacency}
          \mathcal{A}_{ij} = \left\{ \begin{array}{ll}
	0 & \text{if } i = j \\
	2c_{j} - 4\sigma \left(A_{\bolddot,j}\right)^\top b & \text{if } i=0 \text{ and } i \neq j \\ 
	2c_{i} - 4\sigma \left(A_{\bolddot,i}\right)^\top b & \text{if } j=0 \text{ and } j \neq i \\ 
	4F_{i,j} + 4\sigma  \left(A_{\bolddot,j}\right)^\top
        A_{\bolddot, i} \qquad \quad & \text{if } 1 \le i, j \le n
        \text{ and } i \neq j \end{array} \right.
          \end{equation}

    We now state a theorem that allows us to obtain the solution to
    the constrained problem via an unconstrained one. This theorem
    is the key of the algorithm developed afterwards.
	\begin{theorem}\label{thm:PARAM}
		Consider Problem~\eqref{prob:BQP} and Problem~\eqref{prob:UBQP} with optimal values $f^*$ and $h^*$, respectively. 
                Furthermore, assume we have a threshold parameter $\rho$ and a penalty parameter $\sigma$, satisfying the following conditions:
		\begin{enumerate}
			\item\label{thm:a1} Problem~\eqref{prob:BQP} has no feasible solution with value bigger than the threshold $\rho$;
			\item\label{thm:a2} given any vector $x \in \Delta^c$, the value of the penalized function $h(x) = f(x) + \sigma  \lVert Ax-b \rVert^2$ exceeds the threshold parameter $\rho$.
		\end{enumerate}
                Then, for $f^* < +\infty$, $f^*$ is the optimal value of Problem~\eqref{prob:UBQP}, i.e., $h^* = f^*$. Moreover Problem~\eqref{prob:BQP} has no feasible solution if and only if $h^* > \rho$.
	\end{theorem}
	\begin{proof}
                From Remark~\ref{rem:axminusb} it follows that $Ax -
                b = \mathbf{0}$ for $x \in \Delta$ and $\lVert Ax-b \rVert^2 \in
                \Z^+$ for $x \in \Delta^c$. 
                Combining this with the assumptions on the parameters $\rho$ and $\sigma$ we have
		\begin{equation}\label{ineq:HRHO}
		h(x) = \begin{cases} f(x) + 0 \le \rho & \qquad \text{ for } x \in \Delta \\ f(x) + \sigma \lVert Ax - b \rVert^2 > \rho & \qquad \text{ for } x \in \Delta^c \end{cases}
		\end{equation}
		We know that $h^*$ is the minimum of Problem~\eqref{prob:UBQP},
                hence we have $h^* > \rho$ if and only if
                $\Delta = \emptyset$, meaning that
                Problem~\eqref{prob:BQP} is infeasible. 
                On the other hand, if
                $\Delta \neq \emptyset$, then the minimizer of
                Problem~\eqref{prob:UBQP} must lie in the set
                $\Delta$ and
                it follows that $h^* = \min \{ h(x) \mid x
                \in \Delta \} = \min \{f(x) + 0 \mid x \in \Delta \}
                = f^*$.  
    \end{proof}
    
    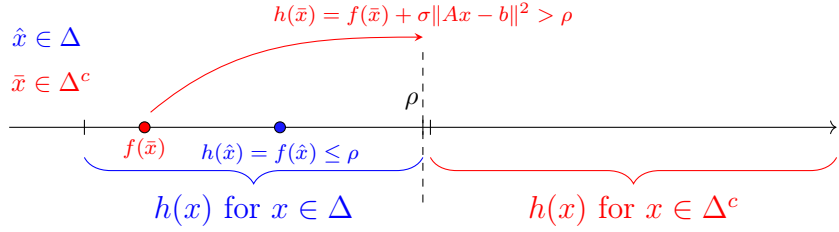
\begin{figure}
    	\begin{center}
    		\begin{tikzpicture}[every edge/.style={shorten <=1pt, shorten >=1pt}]
    		\draw[->] (0,0) to (11,0);
    		\foreach \x/\xtext in {1/,5.5/,5.6/}
    		\draw(\x,0.1)--(\x,-0.1) node[above=5pt] {\xtext};
    		
    		\draw[black,fill=red] (1.8,0) circle (.4ex) node[below,color=red]{\fontsize{8}{12}\selectfont$f(\bar{x})$};
    		\draw[black,fill=red] (5.5,1.8) circle (.0ex) node[below,color=red]{\fontsize{8}{6}\selectfont$h(\bar{x}) = f(\bar{x}) + \sigma \lVert Ax-b \rVert^2 > \rho$};
    		\coordinate (f') at (1.88,0.2);
    		\coordinate (h') at (11,0.1);
    		\coordinate (after) at (5.5,1.2);
    		\draw[->,>=stealth,color=red](f') to [out=40,in=180] node[above]{} (after);
    		
    		\coordinate (l) at (1,-0.4);	
    		\coordinate (u) at (5.5,-0.4);
    		\draw[decorate,decoration={brace,amplitude=10pt,mirror},color=blue] (l) -- (u) node [midway,below=10pt]{$h(x) \text{ for } x \in \Delta$};
    		
    		\draw [dashed] (5.5,-1) -- (5.5,1);	
    		\node at (5.35,0.35) {\footnotesize$\rho$};
    		
    		\node at (0.5,1.2) [color=blue] {\footnotesize$\hat{x} \in \Delta$};
    		\node at (0.56,0.6) [color=red]  {\footnotesize$\bar{x} \in \Delta^c$};
    		
    		\coordinate (l+s) at (5.6,-0.4);
    		\coordinate (fin) at (11,-0.4);
    		\draw[decorate,decoration={brace,amplitude=10pt,mirror},color=red] (l+s) -- (fin) node [midway,below=10pt]{$h(x) \text{ for } x \in \Delta^c$};
    		
    		\draw [black, fill=blue!90!white] (3.6,0) circle (.4ex) node[below=2pt,color=blue]{\fontsize{8}{12}\selectfont$h(\hat{x})=f(\hat{x}) \le \rho$};
    		\end{tikzpicture}
    		 \caption{The values of $h(x)$ for $x \in \Delta$ and $x \in \Delta^c$ are separated by $\rho$.}
    		 \label{fig:thmPARAM}
    	\end{center}
    \end{figure}

    Figure~\ref{fig:thmPARAM} illustrates the role of the parameters
    $\rho$ and $\sigma$:
    if $x \in \Delta$ we have $h(x) = f(x) \le \rho$. On the other
    hand, if $x \in \Delta^c$, adding the penalty term to $f(x)$
    yields $h(x) = f(x) + \sigma \lVert Ax - b \rVert^2 > \rho$.

	Having such a pair of parameters at hand,
        Theorem~\ref{thm:PARAM} allows us to formulate an \emph{exact
          penalty method over discrete sets} which we call
        \expedis\ and outline in Algorithm~\ref{alg:expa}. 
	
	\begin{algorithm}
          \caption{\small Scheme of an exact penalty method over discrete sets} 
          \label{alg:expa} 
                \TitleOfAlgo{\expedis}
		\KwData{$\hat{F} \in \R^{n \times n}, \hat{c} \in \R^n, \hat{A} \in \Z^{m \times n}, \hat{b} \in \Z^m$ defining problem $\min \big\{ y^\top \hat{F}y + \hat{c}^\top y \mid \hat{A}y = \hat{b}, \ y \in \{0,1\}^n \big\}$}
		\KwResult{optimal solution or certificate of infeasibilty}
		\medskip
		transform to problem $\min \big\{ x^\top F x + c^\top x + \alpha \mid Ax = b, \ x \in \hs \big\}$\;
		compute a {\bf threshold parameter} $\rho$\nllabel{rho}\;
		compute a {\bf penalty parameter} $\sigma$\nllabel{sigma}\;
		set up the {\bf max-cut} problem as given in~\eqref{prob:MC}\;
		solve the max-cut problem giving optimal value $h^*$\nllabel{mc}\;
		\eIf{$h^* > \rho $}{
			problem infeasible\;}
		{transform the optimal cut to the optimal solution of the $0/1$ problem\;}
	\end{algorithm}

	Algorithm \expedis\ reformulates a constrained binary quadratic
	problem into a max-cut instance. The solution of the
        max-cut problem either gives a certificate for infeasiblity of
        the original problem or provides the optimal solution.
        Taking a closer look at the computations in \expedis, all
        steps beside Steps~\ref{rho}, \ref{sigma}, and~\ref{mc} are
        straightforward and computationally cheap.

        To perform Step~\ref{mc}, which is solving the max-cut problem,
        we will use the solver BiqMac (see
        Section~\ref{sec:biqmac}).  

        The description on how to perform Steps~\ref{rho}
        and~\ref{sigma} is given in Section~\ref{sec:IMP}, after we
        develop in Section~\ref{sec:PARAM} the neceessary conditions on choosing $\rho$
        and $\sigma$.

\section{Relation to the Work of Lasserre}\label{sec:LAS}
        Lasserre~\cite{La16} showed that solving
        problem~\eqref{prob:BQP01} is equivalent to minimizing a quadratic
        form in $n+1$ variables on the hypercube $\hsp$.
	In this section we show that this work of Lasserre falls into our concept of an exact penalty method over discrete sets.  	
        In fact, we will show that the choice of the parameters $\rho_\Las$
        and $\sigma_\Las$ in~\cite{La16} satisfies the assumptions of Theorem~\ref{thm:PARAM}. 

        The parameters $\rho_\Las$ and $\sigma_\Las$ in~\cite{La16} are defined
        using the
        minimum and the maximum of the standard SDP relaxation (ignoring
        the linear constraints $Ax=b$), i.e.,
 	\begin{subequations}\label{eqn:r1r2}
	\begin{align}
	\hat{\ell} = \min &\left\{ \langle F,X \rangle + c^\top x + \alpha \mid \begin{bmatrix}	1 & x^\top \\ x & X 	\end{bmatrix} \psd, \ \diag(X) = e \right\},\label{eq:r1}\\
	\hat{u} = \max &\left\{ \langle F,X \rangle + c^\top x + \alpha \mid \begin{bmatrix}	1 & x^\top \\ x & X 	\end{bmatrix} \psd, \ \diag(X) = e \right\}.\label{eq:r2}
	\end{align}
	\end{subequations}
	The threshold and the penalty parameter are defined as
        $\rho_{\Las} = \max \big\{\lvert \hat{\ell} \rvert,
        \lvert \hat{u} \rvert \big\}$ and $\sigma_{\Las} = 2
        \cdot \max \big\{ \lvert \hat{\ell} \rvert, \lvert
        \hat{u}\rvert \big\} + 1$, respectively.
        \begin{lemma}\label{lem:LAS}
          The parameters $\rho_{\Las}$ and $\sigma_{\Las}$ satisfy the assumptions of Theorem~\ref{thm:PARAM}.
        \end{lemma}
        \begin{proof}
		Since Problem~\eqref{eq:r2} is a relaxation of $\max
                \big\{ f(x) \mid x \in \Delta \big\}$,  it is easy to see that 
		\begin{equation*}
		\rho_{\Las} = \max \big\{ \lvert \hat{\ell} \rvert, \lvert \hat{u} \rvert \big\} \geq \lvert \hat{u} \rvert \geq \hat{u} \geq\max \big\{ f(x) \mid x \in \Delta \big\}.
		\end{equation*}
		Hence, every feasible solution of
                Problem~\eqref{prob:BQP} is bounded by $\rho_{\Las}$,
                which is assumption~\ref{thm:a1} of Theorem~\ref{thm:PARAM}.

		Assume now $x \in \Delta^c$, then the penalty added is
                at least $\sigma_{\Las}$ because, by Remark~\ref{rem:axminusb}, $\lVert Ax - b
                \rVert^2 \in \Z^+$.
                Therefore, and by the definition of $\hat{\ell}$, it
                follows that
		\begin{equation*}
		h(x) \geq \hat{\ell} + \sigma_{\Las} = \hat{\ell} + 2 \cdot \max \left\{ \lvert \hat{\ell} \rvert, \lvert \hat{u} \rvert \right\} + 1 \geq \max \left\{ \lvert \hat{\ell} \rvert, \lvert \hat{u} \rvert \right\} + 1 > \rho_{\Las}
		\end{equation*}  
		and thus the parameters $\rho_{\Las}$ and $\sigma_{\Las}$
                satisfy also assumption~\ref{thm:a2} of Theorem~\ref{thm:PARAM}.
	\end{proof}

    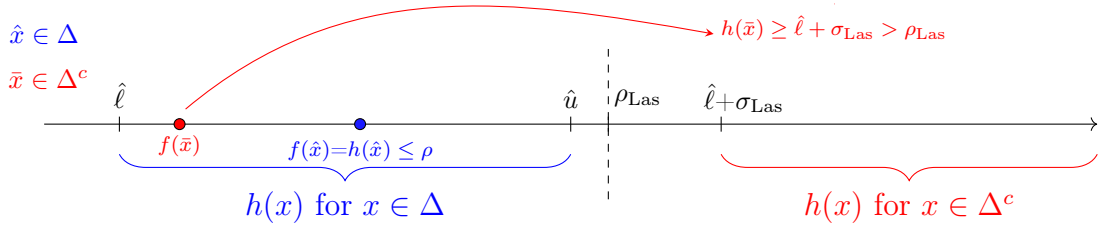
\begin{figure}
		\begin{center}
			\begin{tikzpicture}[every edge/.style={shorten <=1pt, shorten >=1pt}]
			\draw[->] (-1,0) to (13,0);
			\foreach \x/\xtext in {0/\footnotesize$\hat{\ell}$,6/\footnotesize$\hat{u}$,6.5/,8/}
			\draw(\x,0.1)--(\x,-0.1) node[above=5pt] {\xtext};
			
			\draw[black,fill=red] (0.8,0) circle (.4ex) node[below,color=red]{\fontsize{8}{12}\selectfont$f(\bar{x})$};
			\draw[black,fill=red] (9.5,1.6) circle (.0ex) node[below,color=red]{\fontsize{8}{6}\selectfont$h(\bar{x}) \ge \hat{\ell} + \sigma_{\Las} > \rho_{\Las}$};
			\coordinate (f') at (0.88,0.2);
			\coordinate (h') at (11,0.1);
.			\coordinate (after) at (7.9,1.2);
			\draw[->,>=stealth,color=red](f') to [out=40,in=170] node[above] {}  (after);
			
			\coordinate (l) at (0,-0.4);	
			\coordinate (u) at (6,-0.4);
			\draw[decorate,decoration={brace,amplitude=10pt,mirror},color=blue] (l) -- (u) node [midway,below=10pt]{$h(x) \text{ for } x \in \Delta$};
			
			\draw [dashed] (6.5,-1) -- (6.5,1);	
			\node at (6.89,0.35) {\footnotesize$\rho_{\Las}$};
			\node at (8.3,0.35) {\footnotesize$\hat{\ell}$+$\sigma_{\Las}$};
			
			\node at (-1,1.2) [color=blue] {\footnotesize$\hat{x} \in \Delta$};
			\node at (-0.94,0.6) [color=red]  {\footnotesize$\bar{x} \in \Delta^c$};
			
			\coordinate (l+s) at (8,-0.4);
			\coordinate (fin) at (13,-0.4);
			\draw[decorate,decoration={brace,amplitude=10pt,mirror},color=red] (l+s) -- (fin) node [midway,below=10pt]{$h(x) \text{ for } x \in \Delta^c$};
			
			\draw [black, fill=blue!90!white] (3.2,0) circle (.4ex) node[below=2pt,color=blue]{\fontsize{8}{12}\selectfont$f(\hat{x})$=$h(\hat{x}) \leq \rho$};
			\end{tikzpicture}
			\caption{The values of $h(x)$ for $x\in \{-1,1 \}^{n}$.}
			\label{fig:thm}
		\end{center}
	\end{figure}
	Figure~\ref{fig:thm} depicts the proof of Lemma~\ref{lem:LAS}: if $x \in \Delta$, then $\lVert Ax-b \rVert^2 = 0$, thus $h(x) = f(x) \leq \rho_{\Las}$; for $x \in \Delta^c$ the penalty added is at least $\sigma_{\Las}$, thus $h(x) \geq \hat{\ell} + \sigma_{\Las} > \rho_{\Las}$.
	
	Combining the above lemma and Theorem~\ref{thm:PARAM}, we can now restate~\cite[Theorem~2.2]{La16} as 
	\begin{corollary}\label{thm:LAS}
		Consider Problem~\eqref{prob:BQP} and
                Problem~\eqref{prob:UBQP} with optimal values $f^*$ and
                $h^*$, respectively, the threshold parameter
                $\rho_{\Las} = \max \big\{\lvert \hat{\ell}
                \rvert,\lvert \hat{u} \rvert \big\}$ and the penalty
                parameter $\sigma_{\Las} = 2
                \cdot \max \big\{ \lvert \hat{\ell} \rvert, \lvert
                \hat{u}\rvert \big\} + 1$.

                If $f^* < +\infty$, then it follows $h^* = f^*$. Moreover~\eqref{prob:BQP} has no feasible solutions if and only if $h^* > \rho_{\Las}$.
	\end{corollary}
	
	Summarizing, by solving two semidefinite programs with
        variables $X\in \Sy$ and $x\in \R^n$, we can define a
        threshold and a penalty parameter satisfying the assumptions in
        Theorem~\ref{thm:PARAM} and thus apply Algorithm~\expedis.
	
	\section{Conditions on the Threshold and the Penalty
          Parameter}\label{sec:PARAM}
        In Section~\ref{sec:LAS} we \rew{proved} that the parameters chosen
        in~\cite{La16} are a particular choice for $\rho$ and $\sigma$
        in our algorithm \expedis. 
        In this section we investigate necessary conditions on the
        parameters to satisfy the assumptions of
        Theorem~\ref{thm:PARAM} in order to give a wider choice on
        computing $\rho$ and $\sigma$.
        Large penalty parameters can lead to huge numbers in the
        Laplacian of the graph which in turn can have negative effects
        on the computational time for finding the maximum cut.
        Hence, we aim in finding small
        parameters $\rho$ and $\sigma$, still large enough to satisfy
        the assumptions of Theorem~\ref{thm:PARAM}.
        To this end, we define 
	\begin{subequations}\label{eqn:lstarustar}
        \begin{align}
          \ell^* &= \min \big\{ f(x) \mid x \in \Delta^c \big\}, \label{eqn:lstar}\\
          u^* &= \max \big\{ f(x) \mid x \in \Delta \big\}. \label{eqn:ustar}
        \end{align}
	\end{subequations}

	\begin{observation}
	  If $u^* < \ell^*$, then $f(\hat{x}) < f(\bar{x})$ for any
          $\hat{x} \in \Delta$ and any $\bar{x} \in \Delta^c$. Hence,
          any minimizer over $\hs$ will satsify $Ax=b$, i.e.,
          $f^* = \min
          \left\{ f(x) \mid x \in \hs \right\}$ and 
          we can simply ignore all the equality
          constraints because minimization forces $Ax=b$ to hold.
	\end{observation}
        Due to this observation, from now on we assume $\ell^* < u^*$ throughout
        this paper.
        
	\begin{lemma}\label{lem:PARAMbest}
		The parameters $\rho^* = u^*$ and $\sigma^* = u^*
                - \ell^* + \epsilon$ satisfy the assumptions of Theorem~\ref{thm:PARAM}.
	\end{lemma} 

	\begin{proof}
	  To satisfy assumption~\ref{thm:a1} in
          Theorem~\ref{thm:PARAM}, the threshold parameter $\rho$ must
          be an upper bound on the feasible values of
          Problem~\eqref{prob:BQP}. 
          Hence it follows $\rho \geq u^*$. Since there are no
          other constraints on $\rho$, the smallest value of a
          threshold parameter satisfying the assumption of 
          Theorem~\ref{thm:PARAM} is $u^*$.
          Thus we set $\rho^* = u^*$.

	  To satisfy assumption~\ref{thm:a2} in
          Theorem~\ref{thm:PARAM}, 
	  we have to show that $h(x) > \rho^*$ for all $x \in
          \Delta^c$. Let $x \in \Delta^c$, thus
          $\lVert Ax - b\rVert^2$ is a nonnegative
          integer (see Remark~\ref{rem:axminusb}) and
          the penalization added is at least $\sigma$,
          hence $h(x) \geq \ell^* + \sigma$.
          By setting $\sigma^* = u^* - l^* + \epsilon$ it
          follows  
	  \[h(x) \geq \ell^* + \sigma^* = \ell^* + u^* -
          \ell^* + \epsilon = \rho^* + \epsilon > \rho^*. \]
          Thus assumption~\ref{thm:a2} is satisfied. 
	\end{proof}
        
	The value $\sigma^*$ cannot be further decreased, i.e., it is
        the smallest possible formulation of the penalty parameter in
        order to have the assumptions of Theorem~\ref{thm:PARAM}
        satisfied.
	
	\begin{proposition}\label{ex:TIGHT}
		Let $\ell^*$ and $u^*$ be the bounds defined in~\eqref{eqn:lstarustar} and assume the penalty parameter to be $\hat{\sigma} = u^* - \ell^*$. Then there exist binary quadratic problems for which~\eqref{prob:UBQP} is minimized by some vector $x \in \Delta^c$, while $\Delta$ is nonempty. 
	\end{proposition}
	\begin{proof}
		Let the parameters be $\hat{A} = 1$, $\hat{b} = 1$, $\hat{c} = 2$ and $\hat{F}=0$, i.e., the problem is one-dimensional. The transformed parameters are $A = 0.5$, $b= 0.5$, $c = 1$ and $F = 0$, while the additive constant is $\alpha = 1$. Let $x = 1$ be the (unique) optimal solution and $y = \unaryminus 1$ be a (the unique) point in $\Delta^c$. It follows $f(x) = 2$ and $f(y) = 0$. It is easy to see that $\ell^* = \min \big\{f(x) \mid x \in \Delta^c \big\} = 0$ and $u^* = \max \{f(x) \mid x \in \Delta \} = 2$. Assuming $\hat{\sigma} = u^* - \ell^* = 2$, it follows that $h(x) = f(x) + 0 = 2$ and $h(y) = f(y) + \hat{\sigma} \lVert Ax-b \rVert^2 = 2$. Thus Problem~\eqref{prob:BQP} is feasible but the penalized problem is minimized by a vector in the set $\Delta^c$.
	\end{proof}
	
	Clearly, finding $u^*$ is as hard as solving
        Problem~\eqref{prob:BQP} and computing $\ell^*$ is also out of
        reach.
	But any bounds $\ell \leq \ell^*$ and $u \geq u^*$ also
        give rise to a pair of parameters $\rho$ and $\sigma$ that ensures our
        desired assumptions to hold.
	     
	\begin{proposition}\label{prop:SIGMAOPT}
	  Let $\ell$ and $u$ be a lower and an upper bound,
          respectively, such that $\ell \leq \ell^*$ and $u \geq
          u^*$. Moreover, the parameters and the penalized function
          are defined similar as above, i.e., $\rho = u$, $\sigma = u - \ell + \epsilon$. Then the parameters satisfy the assumptions of Theorem~\ref{thm:PARAM}. 
	\end{proposition}
	\begin{proof}
		Since $u \geq u^*$, we have that $\rho = u$ clearly
                ensures assumption~\ref{thm:a1} to hold. And since
                $\ell \leq \ell^*$, the second assumption also holds,
                by using the same arguments as in the proof of Lemma~\ref{lem:PARAMbest} for $\sigma = u - \ell + \epsilon$. 
	\end{proof}

       \subsection{Comparison to the Parameters of Lasserre}
       We now compare the choices of the parameters $\rho$ and $\sigma$ given 
        in Lemma~\ref{lem:LAS} and Lemma~\ref{lem:PARAMbest}.
        Note that the values of the parameters $\rho$ and $\sigma$ are obtained
        through some bounds $\ell$ and $u$. In~\cite{La16}, $\hat{\ell}$ and
        $\hat{u}$ as defined in~\eqref{eqn:r1r2} are used.
        However, Proposition~\ref{prop:SIGMAOPT}, shows that any $\ell\le \ell^*$ 
        and any $u \ge u^*$
        give rise to valid parameters $\rho$ and $\sigma$.
        In order to compare the two different formulations of $\rho$
        and $\sigma$, we fix a pair ($\ell$, $u$). 
	The following proposition shows that our choice of parameters
        is always less (or equal) to the ones proposed by Lasserre~\cite{La16}.
	
	\begin{proposition}\label{prop:PAR+}
		Let $\ell$ and $u$ be any pair of lower and upper bounds such that $\ell\le \ell^*$ and $u \ge u^*$.
                We denote the Lasserre and our new formulations of the
                parameters as follows.
                \[\begin{array}{ll}
                  \rho_{\Las} = \max \{\lvert \ell \rvert, \lvert u
                  \rvert \}  &\sigma_{\Las} = 2\cdot \max \{\lvert \ell \rvert,
                  \lvert u \rvert \} + 1\\
                  \rho_{\GW} = u & \sigma_{\GW} = u - \ell + \epsilon
                \end{array}\]
                Then $\rho_{\GW} \leq \rho_{\Las}$ and $\sigma_{\GW} < \sigma_{\Las}$.
	\end{proposition}
	\begin{proof} The first inequality holds since 
          $\rho_{\GW} = u \leq \lvert u \rvert \leq \max \{\lvert \ell
          \rvert , \vert u \rvert\} = \rho_{\Las}$.

          And the following arguments prove the second inequality:
          $\sigma_{\GW} = u - \ell + \epsilon < \lvert u-\ell \rvert + 1 \leq \lvert u \rvert + \lvert \ell \rvert + 1 \leq 2 \cdot \max \{\lvert u \rvert, \lvert \ell \rvert \} + 1 = \sigma_{\Las}.$
	\end{proof}

	\section{Choosing $\ell$ and $u$ efficiently}\label{sec:IMP}

        In Section~\ref{sec:PARAM} we give a recipe for computing
        a pair of parameters $\rho$ and $\sigma$ that satisfies the assumptions of
        Theorem~\ref{thm:PARAM} by using bounds on 
        $\ell^* = \min \{ f(x) \mid x \in \Delta^c \}$ and 
        $u^* = \max \{ f(x) \mid x \in \Delta\}$.
        The bounds $\hat{\ell}$ and $\hat{u}$, as introduced in
        Section~\ref{sec:LAS}, are candidates since clearly 
        $\hat{\ell} \leq \ell^*$ and $\hat{u} \geq u^*$.

        The time for solving the max-cut instance is influenced by the
        penalty parameter $\sigma$. We aim in finding small values for
        $\sigma$ (but sufficiently large to satisfy the assumptions in
        Theorem~\ref{thm:PARAM}) in order to solve the max-cut problem
        in reasonable time. 
        In this section we will present alternatives to compute tight
        bounds on $\ell^*$ and $u^*$. 
        	
    \subsection{Adding Cutting Planes}\label{sec:cuttingplanes}
    
    The bounds $\hat{\ell}$ and $\hat{u}$ defined in~\eqref{eqn:r1r2} are the solution of the standard
    SDP relaxation.
    These bounds can be strengthened
    by adding cutting planes.
    We denote the upper and lower bound computed by solving the
    standard SDP relaxation 
    with the addition of triangle inequalities and, possibly, of a set of 5-clique inequalities (see Section~\ref{sec:sdp}) by $\tilde{\ell}$ and $\tilde{u}$, i.e.,
   \begin{subequations}\label{eq:tilde}
   \begin{align}
   \tilde{\ell} = \min &\left\{ \langle F,X \rangle + c^\top x + \alpha \mid \begin{bmatrix}	1 & x^\top \\ x & X \end{bmatrix} \psd, \ \diag(X) = e, \ X \in \met \cap X_{I} \right\},\label{eq:ltilde}\\
   \tilde{u} = \max &\left\{ \langle F,X \rangle + c^\top x + \alpha \mid \begin{bmatrix}	1 & x^\top \\ x & X 	\end{bmatrix} \psd, \ \diag(X) = e, \ X \in \met \cap X_{I} \right\}.\label{eq:utilde}
   \end{align}
   \end{subequations}
   where $X_I$ is the set of 5-clique inequalities generated by the heuristic procedure described in Section~\ref{sec:sdp}.
   
   \subsection{Including the Constraints $Ax = b$}
        
        The optimizer leading to $u^*$
        is in the set $\Delta$, therefore
        it satisfies the constraints $Ax = b$.
        In order to take these constraints into account when computing a bound on $u^*$, we
        follow an idea introduced by S.~Burer in~\cite{Bu10} and
        add the equality constraints
        to the standard SDP relaxation.
	
	\begin{proposition}~\cite[Proposition~1]{Bu10}\label{prop:SAM}
		Let $Y$ be the matrix $Y = \begin{bmatrix}
		1 & x^\top \\ x & X \end{bmatrix}$ and suppose $Y \psd$. Moreover let $M$ be defined as $M = [b, \ \unaryminus A]$. Then the following are equivalent:
		\begin{enumerate}
			\item\label{th:1} $Ax = b$ and $\diag(AXA^\top) = b^2$
			\item\label{th:2} $MYM^\top = \mathbf{0}$
			\item\label{th:3} $MY = \mathbf{0}$
		\end{enumerate}
	\end{proposition}
	\begin{proof}
		First we show that $\ref{th:1} \implies \ref{th:2}.$ We have $[b_i, \ \unaryminus A_{i,\bolddot}]Y[b_i, \ \unaryminus A_{i,\bolddot}]^\top = b_i^2 - 2 b_i A_{i,\bolddot} x + A_{i,\bolddot} X A_{i,\bolddot}^\top = b_i^2 - 2 b_i^2 + b_i^2 = 0$. Considering all the rows of $Ax = b$ it follows that $\diag(MYM^\top) = \mathbf{0}$. Since $Y \psd$, $MYM^\top \psd$. Thus $MYM^\top = \mathbf{0}$.
		
		Next we prove $\ref{th:2} \implies \ref{th:3}.$ Since $Y \psd$, let $Y = VV^\top$ be its Gram representation. Then, it follows $0 = \tr(MYM^\top) = \tr(MVV^\top M^\top) = \lVert MV \rVert^2$. Thus $MV = \mathbf{0}$. Hence $MY = (MV)V^\top = \mathbf{0}V^\top = \mathbf{0}$.  
		
		Now we show $\ref{th:3} \implies \ref{th:1}.$ Let us consider the first column of $MY$. It is the zero vector $\mathbf{0}$, hence we have $\mathbf{0} = (M Y)_{\bolddot,1} = [b , \ \unaryminus A][1 , \ x]^\top = b - Ax$, so $Ax = b$. Moreover $MY = \mathbf{0}$ implies $MYM^\top = \mathbf{0}$, hence $0 = (MYM^\top)_{ii} = [b_i, \ \unaryminus A_{i,\bolddot}] Y [b_i, \ \unaryminus A_{i,\bolddot}]^\top$. Expanding $Y$ we have $b_i^2 -2 b_i A_{i,\bolddot}x + A_{i,\bolddot}XA_{i,\bolddot}^\top$. Since we proved above $Ax = b$, we can rewrite  $0 = b_i^2 -2 b_i A_{i,\bolddot}x + A_{i,\bolddot}XA_{i,\bolddot}^\top = -b_i^2 + (AXA^\top)_{ii}$ because $(AXA^\top)_{ii} = A_{i, \bolddot}XA_{i,\bolddot}^\top$. Thus $\diag(AXA^\top) = b^2$. 
	\end{proof}
	As a consequence of Proposition~\ref{prop:SAM}, we can compute the upper bound $u_\Delta$ by solving the SDP
	\begin{equation}\label{eqn:NEWUB}
	u_\Delta = \max \left\{ \langle F, X \rangle + c^\top x + \alpha \mid \begin{bmatrix}
	1 & \rew{x^\top} \\ x & X	\end{bmatrix} \psd, \ \diag(X) = e, \ MY = \mathbf{0} \right\},
	\end{equation}
	where $\mathbf{0}$ is an $m \times (n+1)$ matrix.
	
        We now take a closer look on the inclusion of 
        the equality constraints $MY =
        \mathbf{0}$. In fact, we can transform SDP~\eqref{eqn:NEWUB} into
        a semidefinite program of smaller dimension and 
        less constraints by projecting out the null space. 
	We define the null space of $M$, denoted $\ns(M)$, 
	as the set of vectors that are mapped to $\mathbf{0}$, i.e., $\ns(M) = \{ x \in \R^n \mid Mx = \mathbf{0} \}$. 
	We restate~\cite[Lemma~1]{Bu10} as 
	\begin{lemma}\label{lem:SAM} 
	  Let $M = [b , \ -A]$ and define the closed, convex cone
          $\mathcal{J} = \{ Z \psd \mid MZM^\top = \mathbf{0}
          \}$.  Then $\mathcal{J} = \{ NPN^\top \mid P \psd
          \}$ where $N$ is the matrix whose columns form an
          orthonormal basis of $\ns(M)$.   
	\end{lemma}
        
		
	
		In order to simplify notation, we define the matrix $F' = \begin{bmatrix} \alpha & (c/2)^\top \\ c/2 & F \end{bmatrix}$ to obtain $\langle F,X \rangle + c^\top x + \alpha = \langle F',Y \rangle$. Reformulating \eqref{eqn:NEWUB} using Lemma~\ref{lem:SAM}, we obtain
	\begin{align*}
	u_\Delta & = \max \left\{ \langle F', Y \rangle \mid Y \psd, \ \diag(Y) = e, \ MY = \mathbf{0} \right\} \\
	& = \max \{ \langle F', NPN^\top \rangle \mid P \psd, \ \diag(NPN^\top) = e \} \\
	& = \max \{ \langle F', NPN^\top \rangle \mid P \psd, \ \langle e_je_j^\top , NPN^\top \rangle = 1, \  j \in \{1,\dots,n+1\} \} \\
	& = \max \{ \langle N^\top F' N, P \rangle \mid P \psd, \ \langle N^\top e_j e_j^\top N , P \rangle = 1, \ j \in \{1,\dots,n+1\} \} \\
	& = \max \{ \langle N^\top F' N, P \rangle \mid P \psd, \ \langle (N_{j,\bolddot})^\top N_{j,\bolddot} , P \rangle = 1, \ j \in \{1,\dots,n+1\} \}.
        \end{align*}
        
	Thus, in order to compute $u_\Delta$, we have to solve an SDP of size $\dim(N) = n+1-\rank(M)$ with $n+1$ constraints.
       \rew{ \begin{observation} In case the SDP for computing $u_\Delta$ is infeasible, clearly Problem~\eqref{prob:BQP} is infeasible as well.
        \end{observation} }

        \section{Refinements of the Algorithm}\label{sec:REF}
	\subsection{Detecting (In)feasiblity}
The threshold parameter
$\rho$ has a key role in Algorithm~\ref{alg:expa}, because it
certifies (in)feasibility of Problem~\eqref{prob:BQP}. An 
important difference between Problems~\eqref{prob:BQP} and~\eqref{prob:MC}
is that the first one can be infeasible, while for the
latter one any $\rew{\{\unaryminus 1,1\}}$ vector is a feasible solution.
We present a condition that allows early stopping of the  
branch-and-bound algorithm for solving
Problem~\eqref{prob:MC} in case of infeasibility of Problem~\eqref{prob:BQP}.

\begin{proposition}
	Let $Q$ be defined as in~\eqref{eq:Q} and let $z_{\mathrm{ub}}$ denote the global upper bound of the max-cut problem throughout the branch-and-bound algorithm. If $z_{\mathrm{ub}} < e^\top Qe - \rho$, then Problem~\eqref{prob:BQP} is infeasible.
\end{proposition}
\begin{proof}
	Let $\zmc = \max \left\{x^\top C x \mid x\in \hsp \right\}$
	where $C=\frac{1}{4}L$ and $L$ is the Laplace matrix of the
	graph with adjacency matrix as given in~\eqref{eq:adjacency}.
	We know that $h^* = e^\top Qe - \zmc$.
	Moreover, by Theorem~\ref{thm:PARAM}, Problem~\eqref{prob:BQP} is infeasible if and only if $h^* > \rho$. Combining these
	two considerations, it follows that $\zmc < e^\top Q e - \rho$ is equivalent to saying that
	Problem~\eqref{prob:BQP} has no feasible solution. By
	definition $\zmc \le z_{\mathrm{ub}}$,
	hence $z_{\mathrm{ub}} < e^\top Q e - \rho$ implies
	infeasibility of the original problem. 
\end{proof}

On the other hand, the following lemma shows that any cut with value above a certain threshold
gives rise to a feasbile solution of Problem~\eqref{prob:BQP}.

\begin{lemma}\label{lem:feasSol}
	Let $z_{\mathrm{lb}}$ be the value of some cut, i.e., a
	lower bound on $\zmc$. If
	$z_{\mathrm{lb}} \geq e^\top Qe - \rho$, we can derive a
	feasible vector $x_{\mathrm{lb}} \in \Delta$ from the cut associated with $z_{\mathrm{lb}}$.
\end{lemma} 
\begin{proof}
	Let $\bar{x}_{\mathrm{lb}}$ be the vector associated with $z_{\mathrm{lb}}$,
	i.e., $z_{\mathrm{lb}} = \bar{x}_{\mathrm{lb}}^\top C \bar{x}_{\mathrm{lb}}$.
	Since $\bar{x}_{\mathrm{lb}} \in \hsp$ and $\bar{x}_{\mathrm{lb}}^\top C \bar{x}_{\mathrm{lb}} = (\unaryminus \bar{x}_{\mathrm{lb}})^\top C (\unaryminus \bar{x}_{\mathrm{lb}})$,
	we can fix the first component of $\bar{x}_{\mathrm{lb}}$ to 1, i.e., $\bar{x}_{\mathrm{lb}} = \begin{bmatrix} 1 \ x_{\mathrm{lb}} \end{bmatrix}^\top$.
	Since $\bar{x}_{\mathrm{lb}}$ is a cut, by the transformations presented in Section~\ref{sec:EXPA} it follows that 
	$h(x_{\mathrm{lb}}) = \bar{x}_{\mathrm{lb}}^\top Q \bar{x}_{\mathrm{lb}} = e^\top Q e - \bar{x}_{\mathrm{lb}}^\top C \bar{x}_{\mathrm{lb}}$.  
	Thus  $z_{\mathrm{lb}} \geq e^\top Qe - \rho$ implies $h(x_{\mathrm{lb}}) \le \rho$.
	By Theorem~\ref{thm:PARAM} it follows that $x_{\mathrm{lb}}$ is a feasible vector for Problem~\eqref{prob:BQP}.   	
\end{proof}

	\subsection{Known Feasible Solutions}\label{sec:FEAS}
        In case a feasible solution of \eqref{prob:BQP} is known, e.g., 
        if the assumption of Lemma~\ref{lem:feasSol} holds, we can use
        this information to update the penalty \rew{parameter}. 
        Obviously, detecting infeasiblity is not needed anymore,
        hence we omit the threshold parameter $\rho$.
        Given $x'\in \Delta$ and $\ell \le \ell^*$, we define the penalty parameter $\sigma'$ as follows.
        \[ \sigma' = f(x') - \ell + \epsilon \]
        Note that $f(x') \le u^*$ and therefore $\sigma'$ is smaller than the penalty parameter
        defined before in Section~\ref{sec:PARAM}.
	\begin{theorem}\label{thm:PARAMFeas}
		Consider Problem~\eqref{prob:BQP} and Problem~\eqref{prob:UBQP} with optimal values $f^*$ and $h^*$, respectively. 
                Furthermore, assume that we have a feasible solution $x' \in \Delta$ and we define the penalty parameter $\sigma' = f(x') - \ell + \epsilon$ where $\ell \le \ell^*$.
                Then $h^* = f^*$. 
	\end{theorem}
        \begin{proof} 
        	Suppose $h^* \neq f^*$. 
        	Let $\tilde{x}$ be the vector minimizing $h(x)$ over the set $\hs$, i.e., $h^* = h(\tilde{x})$.
            For any $x \in \Delta$ the equality $f(x) = h(x)$ holds, 
            hence we have $\tilde{x} \in \Delta^c$.
        	By definition of $h(x)$, $\sigma'$ and  $\ell$ it follows 
        	$h(\tilde{x}) = f(\tilde{x}) + \sigma' \lVert A\tilde{x}-b \rVert^2 \ge \ell + \sigma' = \ell + f(x') - \ell + \epsilon > f(x')$. 
        	Thus $h(x') < h(\tilde{x}) = h^*$, hence a contradiction.
        \end{proof}

    	\subsection{Least Violated Solution}
        In case of an infeasible instance, it is possible to detect
        the point with the least violation by relaxing the condition
        on the parameter $u$. Let $(\ell,u)$ be any pair of \rew{values} such
        that $\ell \leq \ell^*$ and $u \geq \max \{f(x) \mid x \in \hs \}$.  Given $\sigma = u-\ell+ \epsilon$, let 
        $\tilde{x}$ be the minimizer of $h(x)$. 
        Then $\tilde{x} \in \argmin\{ \|Ax - b\|\mid x \in \hs\} $, i.e., the point with the least violation.
        This information can be helpful in case one is interested in some measure of infeasibility. 
        
        \begin{lemma}
          Let $x_1$ and $x_2$ be any two vectors in $\hs$. 
          Furthermore, let $h(x)$ be defined using $\sigma = u-\ell+\epsilon$ where
          $\ell \leq \ell^*$ and $u \geq \max \{f(x) \mid x \in \hs \}$.
          Then, if $h(x_1) \leq h(x_2)$, 
          it follows $\lVert Ax_1 - b \rVert \leq \lVert Ax_2 - b \rVert$.
        \end{lemma}
    \begin{proof}
    	The set $\hs$ is partitioned into $\Delta$ and $\Delta^c$. 
        Let $x_1$ and $x_2$ be any two vectors in $\hs$ with $h(x_1) \leq h(x_2)$.
    	If $x_1 \in \Delta$, then $0 = \lVert Ax_1 - b \rVert \leq \lVert Ax_2 - b \rVert$.
    	If $x_2 \in \Delta$, by Theorem~\ref{thm:PARAM} we have $h(x_1) \leq h(x_2) \leq \rho$ 
    	and therefore $x_1$ is in $\Delta$ as well. 
    	Thus $\lVert Ax_1 - b \rVert = 0 = \lVert Ax_2 - b \rVert$.  
    	Hence, we study the remaining (and only interesting) case, namely $x_1, x_2 \in \Delta^c$.
    	By the definition of $u$ and $\ell$ we have
    	\[ \ell + \sigma \lVert Ax_1 - b \rVert^2 \leq f(x_1) + \sigma \lVert Ax_1 - b \rVert^2 
    	\leq f(x_2) + \sigma\lVert Ax_2 - b \rVert^2 \leq u + \sigma \lVert Ax_2 - b \rVert^2.  \] 
    	Hence it follows $\sigma (\lVert Ax_1 - b \rVert^2 - \lVert Ax_2 - b \rVert^2) \leq u - \ell < \sigma$.
    	Dividing the inequality by $\sigma$ and rearranging the terms we have $\lVert Ax_1 - b \rVert^2 < \lVert Ax_2 - b \rVert^2 + 1$.
    	By Remark~\ref{rem:axminusb} we know that $\lVert Ax - b \rVert^2 \in \Z^+$ for any vector $x \in \hs$, thus $\lVert Ax_1 - b \rVert^2 \leq \lVert Ax_2 - b \rVert^2$.
    \end{proof}

	\section{Experimental Results}\label{sec:RES}

        We implemented the bound computations and the computations of the transformed problem in Matlab. As max-cut solver we use BiqMac, as described in Section~\ref{sec:biqmac}. 
        All experiments were done on an Intel Xeon W-2195~CPU @~2.30GHz 
        and 512~GB RAM running under Linux. 
	
        Several problems can be stated as a BQP with equality constraints. 
        In this section we present results for 
        randomly generated instances,
        for the max $k$-cluster problem and for the quadratic boolean cardinality 
        problem.
        
\subsection{Description of the Instances}\label{sec:INST}
\subsubsection{\rew{Randomly Generated Instances}}
    In order to test the efficiency of our algorithm we created 
    two families of \rew{randomly generated instances, denoted RGI,} using Matlab. 
    Since for binary vectors $y$ the equality $y = y^2$ holds, 
    it is possible to exchange $\hat{c}$ and $\diag(\hat{F})$, 
    hence we assume $ \hat{c} = \mathbf{0}$.
    Given a scalar value $b_v$ we define $\hat{b} = [b_v, \dots , b_v ]^\top$.
    For the choice of $\hat{A}_{ij}$ we pick random integer numbers in the interval $[\hat{A}_l,\hat{A}_u]$. 
    Similarly $\hat{F}_{ij}$ is randomly chosen 
    in $\Z \cap [\hat{F}_l , \hat{F}_u]$.
    We assume $\hat{F}$ to be 	symmetric.

    In the first family we choose
    $[\hat{A}_l, \hat{A}_u]$ as $[\unaryminus 1,1]$, $[\unaryminus 3,3]$ and $[\unaryminus 7,7]$. 
    Similarly, we choose
    $[\hat{F}_l, \hat{F}_u]$ as $[\unaryminus 1,1]$, $[\unaryminus 3,3]$ and $[\unaryminus 7,7]$.
    Moreover, we set $b_v = 0$, hence there is always a feasible vector,
    namely $y = \mathbf{0}$.

    In the second family we choose the elements in $\hat{A}$ to be
    from the intervals $[0,1]$ and $[0,3]$ and
    we choose the scalar $b_v \in \{10, 15, 20\}$.
    For the choice of $\hat{F}$ we use different intervals, 
    namely $[0,5]$, $[\unaryminus 5,5]$, $[0,10]$ and $[\unaryminus 10,10]$.

    For every combination of $\hat{A}, \hat{F}$ and $b_v$, 
    \rew{we form two sets of instances with size $n\in \{80,100\}$}.
    For each of these sets we create 15~instances having one to
    15~constraints.
    
    \rew{In total this gives 270~instances in the first family and
    720~instances in the second family.}
    All the \rew{randomly generated instances} can be downloaded from~\cite{math.aau}.

    \subsubsection{$k$-Cluster Problem and Cardinality Boolean Quadratic Problem}\label{sec:kcluster}
The max $k$-cluster problem, sometimes called densest $k$-subgraph problem asks, given a graph $G$, to find the induced subgraph on $k$ vertices with the largest number of induced edges, i.e.,
\[ \max \left\{ \frac{1}{2} y^\top A  y \mid y \in \{0,1\}^n, \ y^\top e = k \right\} \]
where $A \in \R^{n \times n}$ is the adjacency matrix of the graph and $k$ is an integer number in $[1,n]$.

We use the max $k$-cluster instances from~\cite{lambert} where 
$n \in \{120,140,160\}$, $k \in \{n/4,n/2,3n/4\}$, and densities $d \in \{0.25,0.50,0.75\}$.

A slightly more general problem is the cardinality boolean quadratic problem (CBQP). 
The CBQP is a minimization problem similar to the $k$-cluster problem, with the addition of a linear 
term in the objective function, i.e.,
\[ \min \left\{ y^\top Q  y + q^\top y \mid y \in \{0,1\}^n, \ y^\top e = k \right\}. \]
A collection of CBQP instances can be found at~\cite{cbqp}.
These instances have different sizes $n \in \{50, 75, 100, 200, 300\}$ and densities $d \in 
\{0.10,0.50,0.75,1.00 \}$. Following the study of Grossmann and Lima~\cite{LiGro:17} we set $k=n/5$ and 
$k=4n/5$.

\subsection{Comparing the Penalty Parameters}
In this section we study different penalty parameters obtained 
from the bounds that we introduced in the previous sections. 
First we compare the values of different penalty parameters and their
computational times. 
We compare the following penalty parameters:
\begin{align*}
\sigma_{\Las} &= 2 \max\{\lvert \hat{l} \rvert,\lvert \hat{u} \rvert \}
               + 1\\
\sigma_{\CLI}&= \tilde{u} - \tilde{\ell} + \epsilon\\
\sigma_{\GW} &= u_\Delta - \tilde{\ell} + \epsilon
\end{align*}
as described in Sections~\ref{sec:LAS} and~\ref{sec:IMP}.
In Table~\ref{table:rho} we list in columns~2 and~3 the average of the
ratios $\frac{\sigma_{\CLI}}{\sigma_{\Las}}$ and
$\frac{\sigma_{\GW}}{\sigma_{\Las}}$ in percent.
Columns~4 to~6 give the average time in seconds to compute the relevant $\sigma$.

	\begin{center}
		\begin{table}[ht]
                  \begin{center}
			\begin{tabular}{|c!{\reww\vrule width 0.1pt}!{\reww\vrule width 0.1pt}c|c!{\reww\vrule width 0.1pt}!{\reww\vrule width 0.1pt}c|c|c|}
			\hline
			Set of instances & 
			avg $\frac{\sigma_{\CLI}}{\sigma_{\Las}} \rew{(\%)}$ & 
			avg $\frac{\sigma_{\GW}}{\sigma_{\Las}} \rew{(\%)}$ & 
			avg $t_{\Las}(s)$ & avg $t_{\CLI}$(s) & avg $t_{\GW}$(s) \\
			\cline{1-6}
			\rew{RGI with} $n = 80$ & 41.94 & 15.89 & 0.15 & 126.53 & 31.18 \\
			\cline{1-6}
			\rew{RGI with} $n = 100$ & 25.03 & 13.20 & 0.20 & 142.36 & 45.08 \\
			\cline{1-6}
			$k$-cluster & 49.89 & 37.48 & 0.34 & 16.93 & 7.02 \\	
			\cline{1-6}	
			CBQP with $Q(\ge 0)$ & 50.17 & 17.92 & 0.51 & 28.71 & 63.48 \\	
			\cline{1-6}	
			CBQP with $Q(\in \R)$ & 89.89 & 76.00 & 0.54 & 31.91 & 52.88 \\	
			\hline
			\end{tabular}			
		\caption{Comparison of different penalty parameters and computational times}	
		\label{table:rho}
                \end{center}
		\end{table}
	\end{center}

We already proved in Proposition~\ref{prop:PAR+} that $\sigma_{\Las} \geq \max\{ \sigma_{\GW},
\sigma_{\CLI} \}$,
but there is no relation between $\sigma_{\CLI}$ and $\sigma_{\GW}$. 
From Table~\ref{table:rho} we observe that, for all our instances, on average the latter is always smaller than 
the former one.

As expected, the computational time for computing $\sigma_{\CLI}$ and
$\sigma_{\GW}$ is clearly larger than the one for computing
$\sigma_{\Las}$. However, compared to the time for solving the max-cut
problem, the time for computing $\sigma$ is negligible.  
Comparing the times for computing $\sigma_\CLI$ and $\sigma_\GW$,
there is no clear winner. 

We now show the impact of using a smaller penalty parameter on the overall 
performance of our algorithm. 
To do so, we study the running times of \expedis\ with
different choices for the penalty parameter.
We only compare the two penalty parameters
$\sigma_{\Las}$ and $\sigma_{\GW}$ since $\sigma_{\Las} \ge
\sigma_{\CLI} \ge \sigma_{\GW}$.

Furthermore, we experiment with the effect of updating the penalty
parameter $\sigma_{\GW}$ when a feasible solution is found at the root
node, as described in
Section~\ref{sec:FEAS}.

In Figures~\ref{fig:n80} and~\ref{fig:n100}
we show the results for the \rew{randomly generated instances}
 of size $n=80$ and
$n=100$\rew{, respectively}; we set a time limit of 1.5~hours.

We see that decreasing the value of the penalty
parameter improves the average computational time of the algorithm.
Also, updating the penalty parameter at the root node 
improves the overall running time. 
We also \rew{observe that with} an increasing number of variables, the
effect of a smaller penalty parameter is even more significant.

\begin{figure}[ht]
	\begin{tikzpicture}
	\begin{axis}[legend style={at={(-0.3,0.6)},anchor=north east},
	xtick={0,0.5,1,1.5},
	ytick={0,20,40,60,80,100},
	enlarge x limits=false,
	enlarge y limits={upper},
	xlabel=time (h),
	ylabel= \% of instances solved]
	\addplot+[brown, mark=10-pointed star] table [x=t, y=open]{n80.dat};
	\addlegendentry{\expedis\ with $\sigma_{\Las}$}
	\addplot+[red] table [x=t, y=npen]{n80.dat};
	\addlegendentry{\expedis\ with $\sigma_{\GW}$}
	\addplot+[blue, mark options={fill=blue}] table [x=t, y=heur]{n80.dat};
	\addlegendentry{\expedis\ \rew{with~$\sigma$~update}}	    
	\end{axis}
	\end{tikzpicture}
	\caption{Performance profile \rew{of different versions of~\expedis\ on}
		all the \rew{495~randomly generated instances} with $n=80$}
	\label{fig:n80}
\end{figure}
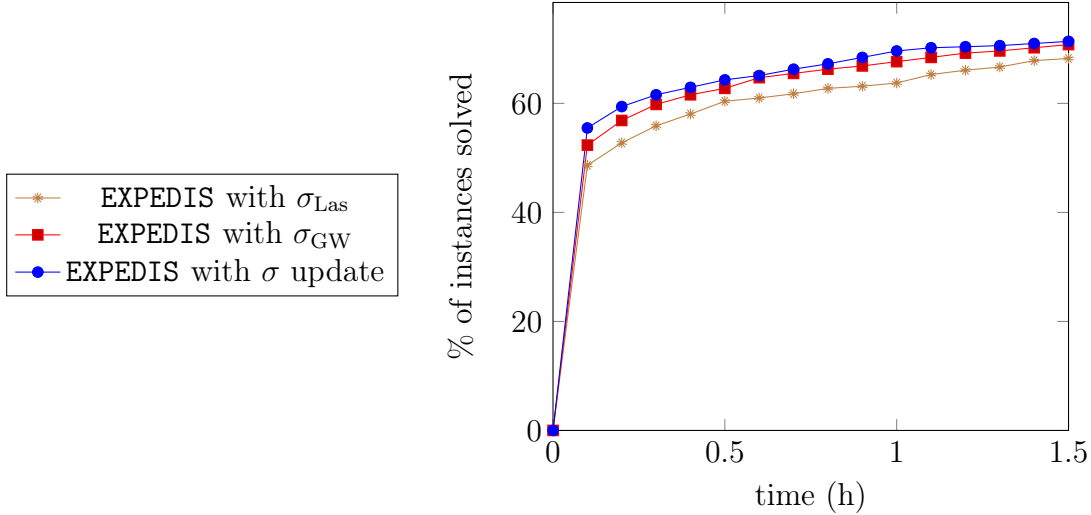

\begin{figure}[ht] 
	\begin{tikzpicture}
	\begin{axis}[legend style={at={(-0.3,0.6)},anchor=north east},
	xtick={0,0.5,1,1.5},	
	ytick={0,20,40,60,80,100},
	enlarge x limits=false,
	enlarge y limits={upper},
	xlabel=time (h),
	ylabel= \% of instances solved]
	\addplot+[brown, mark=10-pointed star] table [x=t, y=open]{n100.dat};
	\addlegendentry{\expedis\ with $\sigma_{\Las}$}
	\addplot+[red] table [x=t, y=npen]{n100.dat};
	\addlegendentry{\expedis\ with $\sigma_{\GW}$}
	\addplot+[blue, mark options={fill=blue}] table [x=t, y=heur]{n100.dat};
	\addlegendentry{\expedis\ \rew{with~$\sigma$~update}}
	\end{axis}
	\end{tikzpicture}
	\caption{Performance profile \rew{of different versions of~\expedis\ on} 
		all the \rew{495~randomly generated instances} with $n=100$}
	\label{fig:n100}
\end{figure}
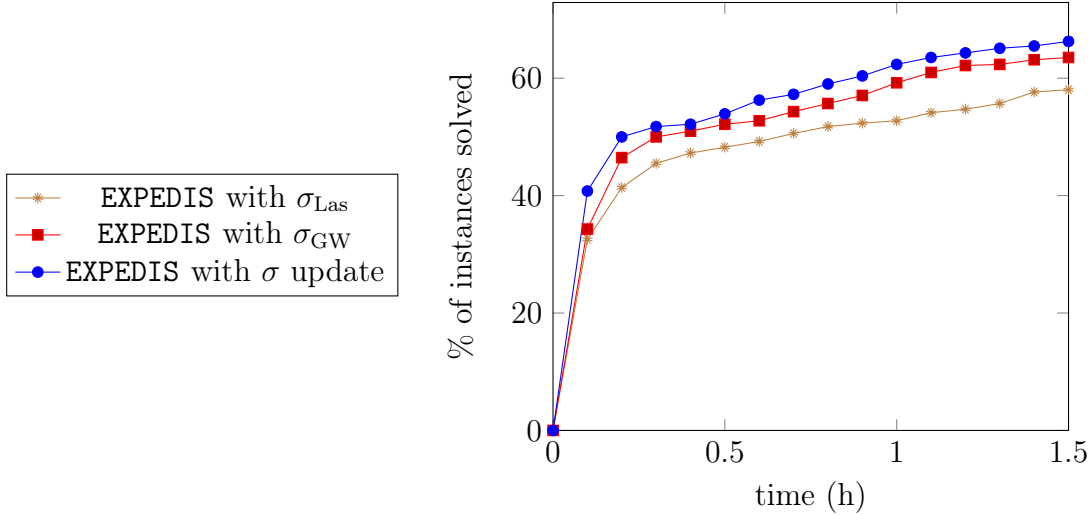

        \subsection{Settings of $\expedis$}\label{sec:setExp}
        The experiments in the previous section suggest to choose in our algorithm 
        $\sigma_{\GW}$, with a possible update \rew{if some feasible solution is known}, 
        as penalty parameter.
        \rew{In order to possibly avoid the preprocessing calculation
          of the parameters, we enhance the algorithm as follows.
        	We first use a trivial penalty parameter to set up the max-cut problem as
        	given in~\eqref{prob:MC}. Then we run the Goemans-Williamson heuristic 
        	and we locally improve the resulting cut vector by checking all possible 
        	moves of a single vertex to the opposite partition block.
        Given as input a feasible problem, 
        this heuristic often finds a cut whose vector 
          associated  with \ref{prob:BQP} is feasible.}

        \rew{Let $x'$ be the vector associated to the cut found by 
        the rounding heuristic. 
        If $x' \notin \Delta$, we set the threshold and the penalty parameter 
        as described in Section~\ref{sec:PARAM}, i.e., 
        \[
	        \rho_{\GW} = u_\Delta \; \text{ and } \; 
	        \sigma_{\GW} = u_\Delta - \tilde{\ell} + \epsilon, 
	    \] 
	    where $\tilde{\ell}$ and $u_\Delta$ are the bounds 
	    presented in Section~\ref{sec:IMP}.
	    If $x' \in \Delta$, i.e., $x'$ is a feasible vector for Problem~\eqref{prob:BQP}, 
        we redefine the penalty parameter as described in Section~\ref{sec:FEAS}, 
        i.e.,
        \[\sigma' = f(x') - \tilde{\ell} + \epsilon.\]}
        We outline our settings in Algorithm~\ref{alg:finalExpa}. 
        
        \begin{algorithm}
        	\caption{\small Scheme of our algorithm} 
        	\label{alg:finalExpa} 
        	\TitleOfAlgo{Refined \expedis}
        	\KwData{$\hat{F} \in \R^{n \times n}, \hat{c} \in \R^n, \hat{A} \in \Z^{m \times n}, \hat{b} \in \Z^m$ defining problem $\min \big\{ y^\top \hat{F}y + \hat{c}^\top y \mid \hat{A}y = \hat{b}, \ y \in \{0,1\}^n \big\}$}
        	\KwResult{optimal solution or certificate of infeasibilty}
        	\medskip
        	transform to problem $\min \big\{ x^\top F x + c^\top x + \alpha \mid Ax = b, \ x \in \hs \big\}$\;
        	\rew{choose a trivial penalty parameter and set up the {\bf max-cut} problem as given in~\eqref{prob:MC}\; 
        	run the rounding heuristic and extract the cut $x'$\;
        	\eIf{$x' \notin \Delta $}{
        		compute the {\bf threshold parameter} $\rho_{\GW} = u_\Delta$\;
        		compute the {\bf penalty parameter} $\sigma_{\GW}=u_\Delta-\tilde{\ell}+\epsilon$\nllabel{sigma2}\;
        		set up and solve the {\bf max-cut} problem giving optimal value $h^*$\;
        		\If{$z_{\mathrm{ub}} < e^\top Q e - \rho_{\GW}$}{
        			terminate the algorithm: problem infeasible\;}
        		\If{$h^* > u_\Delta $}{
        			terminate the algorithm: problem infeasible\;}
        		{}}
        		{update the penalty parameter: $\sigma'=f(x')-\tilde{\ell}+\epsilon$\;
        		set up and solve the {\bf max-cut} problem giving optimal value $h^*$\;}
	        transform the optimal cut to the optimal solution of the $0/1$ problem\;}
        \end{algorithm}
        

	\subsection{Comparison to \rew{other solvers}}
	In this section we compare the performance of our algorithm 
	with other generic solvers.
	\rewt{We tested the instances using BiqCrunch~\cite{KrMaRo17}, 
	COUENNE~\cite{COUENNE}, CPLEX~\cite{CPLEX}, GUROBI~\cite{GUROBI}, 
	SCIP~\cite{SCIP} and SMIQP~\cite{BiElLaWi:17}.
	In all these solvers 
	we input Problem~\eqref{prob:BQP01} and we keep the default settings.}
	We use the data sets described in Section~\ref{sec:INST} above.

	In Figure~\ref{fig:ran80+100} we present the performance
    profile of \rew{the different solvers on} all the \rew{randomly
    generated instances.}  
	Clearly, for these instances~\expedis\ outperforms \rew{all the other solvers}. 
	\rewt{Within the time limit~\expedis\ and GUROBI solve almost 70\% and 60\% of the instances, respectively, BiqCrunch solves slightly less than 50\% of them, 
	while the other solvers less than 30\% of them.}
	
	\begin{figure}[ht]
		\begin{tikzpicture}
		\begin{axis}[legend style={at={(-0.3,0.6)},anchor=north east},
		xtick={0,0.5,1,1.5},
		ytick={0,20,40,60,80,100},
		enlarge x limits=false,
		enlarge y limits={upper},
		xlabel=time (h),
		ylabel= \% of instances solved]
	\addplot+[blue, mark options={fill=blue}] table [x=TIME, y=HEUR]{review2.dat};
	\addlegendentry{\expedis}
	\addplot+[green, mark=diamond*, mark options={fill=green}] table [x=TIME, y=CP]{review2.dat};
	\addlegendentry{CPLEX}
	\addplot+[cyan, mark=asterisk, mark options={fill=cyan}] table [x=TIME, y=SC]{review2.dat};
	\addlegendentry{SCIP}
	\addplot+[red, mark=10-pointed star, mark options={fill=red}] table [x=TIME, y=GU]{review2.dat};
	\addlegendentry{GUROBI}
	\addplot+[black, mark=triangle*, mark options={fill=black}] table [x=TIME, y=COU]{review2.dat};
	\addlegendentry{COUENNE}
	\addplot+[yellow, mark options={fill=yellow}] table [x=TIME, y=SMIQP]{review2.dat};
	\addlegendentry{\rewt{SMIQP}}
	\addplot+[gray, mark=pentagon*, mark options={fill=gray}] table [x=TIME, y=BC]{review2.dat};
	\addlegendentry{\rewt{BIQCRUNCH}}
		\end{axis}
		\end{tikzpicture}
		\caption{Performance profile \rew{of the different solvers} on all 
		the \rew{randomly generated instances}}
		\label{fig:ran80+100}
	\end{figure}
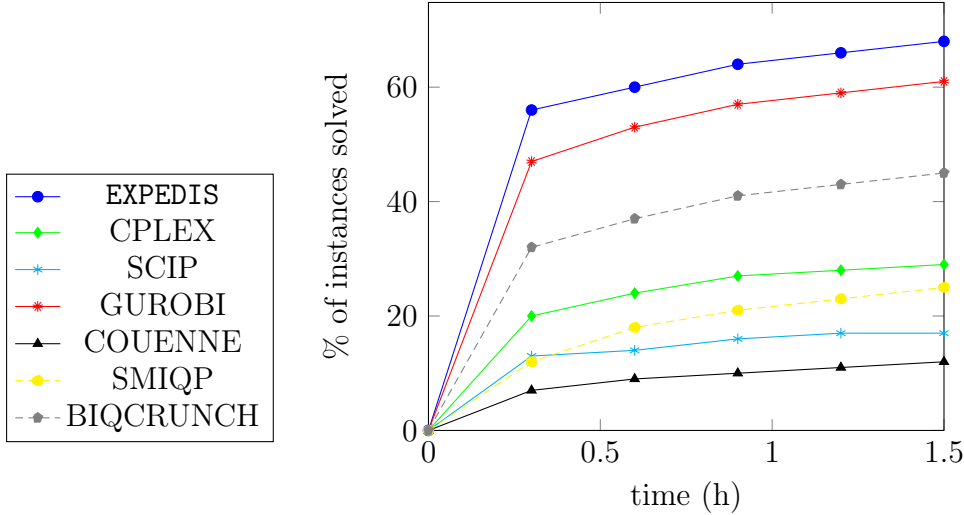

	\rewt{Next, we compare the performance of the different solvers for structured instances,
	namely the max $k$-cluster instances.}
	
	\rewt{From the literature, two prominent solvers for the max $k$-cluster
	problem are BiqCrunch and SMIQP. 
	The former has a tailored version for the max $k$-cluster problem,
	which reinforces the model with additional product constraints.
	The specialized version of BiqCrunch outperforms~\expedis~on the max $k$-cluster
	instances, while the comparison with SMIQP and the general version 
	of BiqCrunch is clearly dominated by~\expedis.}

    \rewt{We report only the running times of \expedis\; on the 135~instances of
    the max $k$-cluster problem as described in Section~\ref{sec:kcluster}
    in Table~\ref{tab:kCLUSTER} below, since none of the other solvers 
      was able to solve any of these instances within the time limit of 3~hours. }
    
    For every combination $(n,k)$ there are 15 instances.
    In the third column of Table~\ref{tab:kCLUSTER}, we present the average running 
    time over the instances solved within the time limit of 3~hours. 
    Note that all the instances with $k=3n/4$ are solved within 
    3~hours and that the problem gets more complicated for small $k$.


   \begin{table}[ht]
	   \begin{center}
         \begin{tabular}{|c|lr|}
        	\hline
        	$(n,k)$ & instances & avg time (s) \\
        	\hline
        	$(120,30)$ & 15 (14) & 2394\\
        	\hline  
        	$(120,60)$ & 15 (15) & 1748\\
        	\hline  
        	$(120,90)$ & 15 (15) &  198 \\
        	\hline  
        	$(140,35)$ & 15 (5)  & 5108\\
        	\hline  
        	$(140,70)$ & 15 (12) & 2922\\
        	\hline  
        	$(140,105)$& 15 (15) & 509 \\
        	\hline  
        	$(160,40)$ & 15 (2)  & 7050\\
        	\hline  
        	$(160,80)$ & 15 (6)  & 5704\\
        	\hline  
        	$(160,120)$& 15 (15) & 1145\\ 
        	\hline
        \end{tabular}			
        \caption{Running times for solving the $k$-cluster instances. For each $(n,k)$ we consider 15~instances, in brackets we give the number of instances solved within the \rew{time limit} of 3~hours.}	
       	\label{tab:kCLUSTER}
      \end{center}
       \end{table}

	In Figure~\ref{fig:cbqp} we compare the running times for solving the instances 
	of the cardinality boolean quadratic problem, for which we set a time limit 1.5~hours.
	\rewt{Since CBQP is essentially the same problem as the max $k$-cluster, 
	i.e., a binary quadratic problem with a cardinality constraints,
	and from the considerations on the max $k$-cluster, 
	we do not proceed further in solving the CBQP instances by using SMIQP and BiqCrunch.
	We do not present the results for $n=50$ because all the solvers 
	solve all the small size instances.} 
	Moreover, for the combination of parameters
	$(n,k)\in \{ (200,n/5), (300,n/5), (300,4n/5)\}$ none of
	the \rew{solvers manages to find the optimum of any instance within the	time limit}. 
	Hence we omit these instances as well.
	We show the results in Figure~\ref{fig:cbqp}. Within the time limit, 
	\expedis\; can solve almost all instances \rew{(95 \%)}, whereas \rew{GUROBI 
	manages to solve slightly less than 70\% and all the other solvers at most half of them.}
	
		\begin{figure}
		\begin{tikzpicture}
		\begin{axis}[legend style={at={(-0.3,0.6)},anchor=north east},
		xtick={0,0.5,1,1.5},
		ytick={0,20,40,60,80,100},
		enlarge x limits=false,
		enlarge y limits=upper,
		xlabel=time (h),
		ylabel= \% of instances solved]
		\addplot+[blue, mark options={fill=blue}] table [x=TIME, y=EX]{cbqpN.dat};
		\addlegendentry{\expedis}
		\addplot+[green, mark=diamond*, mark options={fill=green}] table [x=TIME, y=CP]{cbqpN.dat};
		\addlegendentry{CPLEX}
		\addplot+[cyan, mark=asterisk, mark options={fill=cyan}] table [x=TIME, y=SC]{cbqpN.dat};
		\addlegendentry{SCIP}
		\addplot+[red, mark=10-pointed star, mark options={fill=red}] table [x=TIME, y=GU]{cbqpN.dat};
		\addlegendentry{GUROBI}
		\addplot+[black, mark=triangle*, mark options={fill=black}] table [x=TIME, y=COU]{cbqpN.dat};
		\addlegendentry{COUENNE}
		\end{axis}
		\end{tikzpicture}
		\caption{Performance profile \rew{of the different solvers on~400~CBQP instances.}}
		\label{fig:cbqp}
	\end{figure}
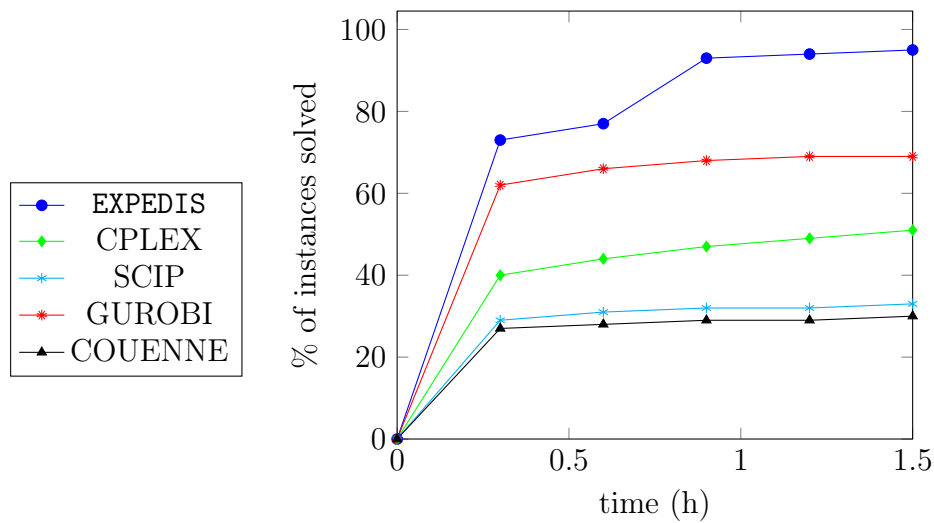

	\section{Conclusions and Future Work}\label{sec:CON}
	In this paper we present \expedis, a new algorithm for solving
        binary quadratic problems with linear equality constraints. 
	\expedis\; transforms the binary quadratic problem into a
        max-cut instance, computes the optimal cut, and then
        provides either the optimal solution of the binary quadratic
        problem or gives a certificate of infeasibility.  

        At the heart of the algorithm is a penalty parameter used for the transformation.
        We investigate conditions on the penalty parameter and present
        different ways to choose it. We also present numerical
        experiments showing the effect of the different choices.

	In order to demonstrate the strength \expedis\; we
        perform numerical experiments on several types of instances. 
        These experiments clearly show the dominance of \expedis\;
        over \rew{CPLEX, COUENNE, GUROBI and SCIP.} 

        Computing the max-cut is done using the solver BiqMac. Hence, advancing 
        BiqMac will also result in a speedup of \expedis. Therefore we
        are currently working on improving BiqMac by adding more
        polyhedral cuts in the bounding procedure.  Another line of
        research is to make a\rew{n extended} parallel version of \expedis\; and run it on
        a high-performance computer. This algorithm will then be
        available via BiqBin \rew{at the web page \url{http://biqbin.eu}}.

	It would be interesting to generalize our algorithm to linear
        inequality constraints as well as to quadratic constraints.
	In particular, we would like to understand whether
        ellipsoidal relaxations~\cite{BuDSPaPi13} can be used to further improve the
        penalty parameter. These topics are currently under investigation.

	\section*{Acknowledgments}\label{sec:ACK}
        This project was supported by the Austrian Science Fund (FWF): I\,3199-N31
        \rew{and by the Slovenian Research Agency (ARRS): N1-0057}.
        Furthermore, this project has received funding from
        the European Union's Horizon~2020 research and innovation programme
        under the Marie Sk\l{}odowska-Curie grant agreement No~764759.
        The authors are grateful to Franz Rendl for helpful
        discussions.  
        Part of this work has been carried out during a
        research stay of the first author at the University of Iowa
        hosted by Sam Burer. 
        The authors would like to thank Sam Burer for discussions and 
        suggestions that improved the results significantly. 
	\rew{The authors also thank two anonymous referees for their valuable comments.}
        
	\bibliographystyle{amsplain}
	\bibliography{expa}

\end{document}